\documentclass{article}%
\usepackage{amsmath}
\usepackage{amsfonts}
\usepackage{amssymb}
\usepackage{graphicx}%
\providecommand{\U}[1]{\protect\rule{.1in}{.1in}}
\newtheorem{theorem}{Theorem}

\newtheorem{corollary}[theorem]{Corollary}

\newtheorem{lemma}[theorem]{Lemma}

\newtheorem{proposition}[theorem]{Proposition}

\newenvironment{proof}[1][Proof]{\noindent\textbf{#1.} }{\ \rule{0.5em}{0.5em}}
\begin{document}

\title{On the order derivatives of Bessel functions }
\author{T. M. Dunster\thanks{Department of Mathematics and Statistics, San Diego State
University, San Diego, CA 92182-7720, U.S.A. \textit{email}:
mdunster@mail.sdsu.edu}}
\maketitle

\begin{abstract}
The derivatives with respect to order for the Bessel functions $J_{\nu}(x)$
and $Y_{\nu}(x)$, where $\nu>0$ and $x\neq0$ (real or complex), are studied.
Representations are derived in terms of integrals that involve the products
pairs of Bessel functions, and in turn series expansions are obtained for
these integrals. From the new integral representations for $\partial J_{\nu
}(x)/\partial\nu$ and $\partial Y_{\nu}(x)/\partial\nu$, asymptotic
approximations involving Airy functions are constructed for the case $\nu$
large, which are uniformly valid for $0<x<\infty$.

\end{abstract}

\section{Introduction}

Due to their importance in mathematics and physics, Bessel functions have been
extensively studied in the literature. Their main properties can be found by
perusing various classic textbooks on the subject ([1], [7], [11], [13],
[16]), and, most recently, [14, Chap. 10]. The purpose of this paper is to
study the order derivatives of Bessel functions. The literature is relatively
sparse on the properties of these functions, with none currently existing on
their uniform asymptotic properties. We remark that the order derivatives are
used in the study of the monotonicity with respect to order of Bessel
functions, which in turn has applications in quantum mechanics [9].

For some integral representations see [3] and [12] (the latter for the special
case $\nu=\pm{\frac{1}{2}})$, and for certain series representations see [4],
[8, \S 8.486(1)], [10, \S 9.4.4] and [15]. We also mention that in [5]
modified Bessel functions play a role in the study of the order derivatives of
Legendre functions

In this paper we derive new integral representations for the derivatives with
respect to order of Bessel and Hankel functions (\S 2), new series expansions
for the integrals appearing in these representations (\S 3), and uniform
asymptotic approximations for the order derivatives of $J_{\nu}\left(
x\right)  $ and $Y_{\nu}\left(  x\right)  $ (\S \S 4 and 5).

\section{Integral representations for order derivatives of Bessel functions}

Our first result is as follows.

\begin{proposition}
\textit{For} $\nu>0$, $\left\vert {\arg\left(  z\right)  }\right\vert \leq\pi
$, \textit{and} $z\neq0$
\begin{equation}
\frac{\partial J_{\nu}\left(  z\right)  }{\partial\nu}=\nu\pi Y_{\nu}\left(
z\right)  \int_{0}^{z}{\frac{J_{\nu}^{2}\left(  t\right)  }{t}dt}+\nu\pi
J_{\nu}\left(  z\right)  \int_{z}^{\infty}{\frac{J_{\nu}\left(  t\right)
Y_{\nu}\left(  t\right)  }{t}dt}. \label{eq1}%
\end{equation}
\textit{If z is complex the paths of integration must lie in }$\mathrm{C}%
\backslash\left(  {-\infty,0}\right]  $\textit{\ (except at the end point for
the first integral), and with the end point at infinity in the second integral
on the positive real axis.}
\end{proposition}

\begin{proof}
For convenience, here and throughout we denote
\[
\hat{{J}}_{\nu}\left(  z\right)  =\partial J_{\nu}\left(  z\right)
/\partial\nu,\ \hat{{Y}}_{\nu}\left(  z\right)  =\partial Y_{\nu}\left(
z\right)  /\partial\nu.
\]
Then, from the $\nu$ derivative of Bessel's equation, we have
\begin{equation}
z^{2}\frac{d^{2}\hat{{J}}_{\nu}\left(  z\right)  }{dz^{2}}+z\frac{d\hat{{J}%
}_{\nu}\left(  z\right)  }{dz}+\left(  {z^{2}-\nu^{2}}\right)  \hat{{J}}_{\nu
}\left(  z\right)  =2\nu J_{\nu}\left(  z\right)  . \label{eq3}%
\end{equation}
We then apply the method of variation of parameters, to obtain
\begin{equation}%
\begin{array}
[c]{l}%
\hat{{J}}_{\nu}\left(  z\right)  =\nu\pi Y_{\nu}\left(  z\right)  \int_{0}%
^{z}{t^{-1}J_{\nu}^{2}\left(  t\right)  dt}+\nu\pi J_{\nu}\left(  z\right)
\int_{z}^{\infty}{t^{-1}J_{\nu}\left(  t\right)  Y_{\nu}\left(  t\right)
dt}\\
+c_{\nu}J_{\nu}\left(  z\right)  +d_{\nu}Y_{\nu}\left(  z\right)  ,
\end{array}
\label{eq4}%
\end{equation}
for certain constants $c_{\nu}$ and $d_{\nu}$. Recalling the behavior of
Bessel functions as $z\rightarrow0_{\,}^{+}$(see for example, [14, Eqs.
(10.7.3, (10.7.4) and (10.15.1)]), we observe that the LHS is bounded, whereas
the RHS is bounded if and only if $d_{\nu}=0$, which we conclude must be the case.

To determine $c_{\nu}$ we compare both sides as $z\rightarrow\infty$. To do
so, we use the well-known approximations [14, \S 10.17(i)]
\begin{equation}%
\begin{array}
[c]{l}%
\left(  {\dfrac{\pi z}{2}}\right)  ^{1/2}J_{\nu}\left(  z\right)  =\cos\left(
{z-{\frac{1}{2}}\nu\pi-{\frac{1}{4}}\pi}\right)  \left\{  {1-\dfrac{\left(
{4\nu^{2}-1}\right)  \left(  {4\nu^{2}-9}\right)  }{128z^{2}}}\right\} \\
-\sin\left(  {z-{\frac{1}{2}}\nu\pi-{\frac{1}{4}}\pi}\right)  \dfrac{4\nu
^{2}-1}{8z}+O\left(  {\dfrac{1}{z^{3}}}\right)  ,
\end{array}
\label{eq5}%
\end{equation}
and
\begin{equation}%
\begin{array}
[c]{l}%
\left(  {\dfrac{\pi z}{2}}\right)  ^{1/2}Y_{\nu}\left(  z\right)  =\sin\left(
{z-{\frac{1}{2}}\nu\pi-{\frac{1}{4}}\pi}\right)  \left\{  {1-\dfrac{\left(
{4\nu^{2}-1}\right)  \left(  {4\nu^{2}-9}\right)  }{128z^{2}}}\right\} \\
+\cos\left(  {z-{\frac{1}{2}}\nu\pi-{\frac{1}{4}}\pi}\right)  \dfrac{4\nu
^{2}-1}{8z}+O\left(  {\dfrac{1}{z^{3}}}\right)  .
\end{array}
\label{eq6}%
\end{equation}
Taking the $\nu$ derivative of (\ref{eq5}), we have (formally) for the LHS of
(\ref{eq4})
\begin{equation}
\hat{{J}}_{\nu}\left(  z\right)  \sim\left(  {\frac{\pi}{2z}}\right)
^{1/2}\sin\left(  {z-{\frac{1}{2}}\nu\pi-{\frac{1}{4}}\pi}\right)
\quad\left(  {z\rightarrow\infty}\right)  . \label{eq7}%
\end{equation}
This can be verified, for example, by a straightforward modification of the
saddle point method found in [13, Chap. 4, \S 9.1] to the integral
representation
\[
\hat{{J}}_{\nu}\left(  z\right)  =-\frac{1}{2\pi i}\int_{-\infty+\pi
i}^{-\infty-\pi i}{t\exp\left\{  {-z\sinh\left(  t\right)  +\nu t}\right\}
dt}\quad\left(  {\left\vert {\arg z}\right\vert <{\frac{1}{2}}\pi}\right)  .
\]
On the other hand, from (\ref{eq5}) and (\ref{eq6}), we find the RHS of
(\ref{eq4}) takes the form
\begin{equation}
\left(  {\frac{\pi}{2z}}\right)  ^{1/2}\sin\left(  {z-{\frac{1}{2}}\nu
\pi-{\frac{1}{4}}\pi}\right)  +c_{\nu}\left(  {\frac{2}{\pi z}}\right)
^{1/2}\cos\left(  {z-{\frac{1}{2}}\nu\pi-{\frac{1}{4}}\pi}\right)  +O\left(
{\frac{1}{z^{3/2}}}\right)  . \label{eq9}%
\end{equation}
In arriving at this we have used [16, p. 405]
\begin{equation}
\int_{0}^{\infty}{\frac{J_{\nu}^{2}\left(  t\right)  }{t}dt}=\frac{1}{2\nu}.
\label{eq10}%
\end{equation}
Comparing (\ref{eq7}) with (\ref{eq9}) we deduce that $c_{\nu}=0$, and the
result (\ref{eq1}) follows.
\end{proof}

\begin{proposition}
Under the same conditions as Proposition 1.1
\begin{equation}
\frac{\partial Y_{\nu}\left(  z\right)  }{\partial\nu}=\nu\pi J_{\nu}\left(
z\right)  \int_{z}^{\infty}{\frac{Y_{\nu}^{2}\left(  t\right)  }{t}dt}-\nu\pi
Y_{\nu}\left(  z\right)  \int_{z}^{\infty}{\frac{J_{\nu}\left(  t\right)
Y_{\nu}\left(  t\right)  }{t}dt}-\frac{1}{2}\pi J_{\nu}\left(  z\right)  .
\label{eq11}%
\end{equation}

\end{proposition}

\begin{proof}
From [14, Eq. (10.22.6)] we have, for $\mu\neq\pm\nu$,
\begin{equation}
\int{\frac{\mathcal{C}_{\mu}\left(  z\right)  \mathcal{D}_{\nu}\left(
z\right)  }{z}dz}=\frac{z\left\{  \mathcal{C}{_{\mu}\left(  z\right)
\mathcal{D}_{\nu+1}\left(  z\right)  -\mathcal{C}_{\mu+1}\left(  z\right)
\mathcal{D}_{\nu}\left(  z\right)  }\right\}  }{\left(  {\mu^{2}-\nu^{2}%
}\right)  }+\frac{\mathcal{C}_{\mu}\left(  z\right)  \mathcal{D}_{\nu}\left(
z\right)  }{\mu+\nu}+C, \label{eq12}%
\end{equation}
where $C$ is some constant. We first choose $\mathcal{C}=J$ and $\mathcal{D}%
=Y$, and we shall select $C$ so that the RHS is finite as $\mu\rightarrow\nu$.
To do so, note from [14, Eq. (10.5.2)]
\begin{equation}
J_{\nu+1}\left(  z\right)  Y_{\nu}\left(  z\right)  -J_{\nu}\left(  z\right)
Y_{\nu+1}\left(  z\right)  =2/\left(  {\pi z}\right)  . \label{eq13}%
\end{equation}
Thus the choice
\[
C=\frac{2}{\pi\left(  {\mu^{2}-\nu^{2}}\right)  }+C_{0}=\frac{z\left\{
{J_{\nu+1}\left(  z\right)  Y_{\nu}\left(  z\right)  -J_{\nu}\left(  z\right)
Y_{\nu+1}\left(  z\right)  }\right\}  }{\mu^{2}-\nu^{2}}+C_{0},
\]
ensures that the RHS of (\ref{eq12}) is indeed finite as $\mu\rightarrow\nu$,
for any bounded $C_{0}$. Thus we can assert that
\begin{equation}%
\begin{array}
[c]{l}%
\int_{z}^{\infty}{\dfrac{J_{\mu}\left(  t\right)  Y_{\nu}\left(  t\right)
}{t}dt}=-\dfrac{zY_{\nu+1}\left(  z\right)  }{\mu+\nu}\left(  {\dfrac{J_{\mu
}\left(  z\right)  -J_{\nu}\left(  z\right)  }{\mu-\nu}}\right) \\
+\dfrac{zY_{\nu}\left(  z\right)  }{\mu+\nu}\left(  {\dfrac{J_{\mu+1}\left(
z\right)  -J_{\nu+1}\left(  z\right)  }{\mu-\nu}}\right)  -\dfrac{J_{\mu
}\left(  z\right)  Y_{\nu}\left(  z\right)  }{\mu+\nu}+C_{0}.
\end{array}
\label{eq15}%
\end{equation}
where $C_{0}$ is now a specific constant which must be determined. To do so,
we let $\mu\rightarrow\nu$, and (\ref{eq15}) then becomes
\begin{equation}
\int_{z}^{\infty}{\frac{J_{\nu}\left(  t\right)  Y_{\nu}\left(  t\right)  }%
{t}dt}=\frac{z}{2\nu}\left[  {\frac{\partial J_{\nu+1}\left(  z\right)
}{\partial\nu}Y_{\nu}\left(  z\right)  -\frac{\partial J_{\nu}\left(
z\right)  }{\partial\nu}Y_{\nu+1}\left(  z\right)  }\right]  -\frac{J_{\nu
}\left(  z\right)  Y_{\nu}\left(  z\right)  }{2\nu}+C_{0}. \label{eq16}%
\end{equation}
Letting $z\rightarrow\infty$ with $\nu$ fixed, we find with the aid of
(\ref{eq5}) and (\ref{eq6}) that $C_{0}=0$. Moreover, it is seen that
\begin{equation}
\int_{z}^{\infty}{\frac{J_{\nu}\left(  t\right)  Y_{\nu}\left(  t\right)  }%
{t}dt}=\frac{\sin\left(  {2z-\nu\pi}\right)  }{2\pi z^{2}}+O\left(  {\frac
{1}{z^{3}}}\right)  \quad\left(  {z\rightarrow\infty}\right)  , \label{eq17}%
\end{equation}
for each fixed $\nu$.

We next interchange $J$ and $Y$ in (\ref{eq16}), to obtain
\begin{equation}
\int_{z}^{\infty}{\frac{J_{\nu}\left(  t\right)  Y_{\nu}\left(  t\right)  }%
{t}dt}=\frac{z}{2\nu}\left[  {\frac{\partial Y_{\nu+1}\left(  z\right)
}{\partial\nu}J_{\nu}\left(  z\right)  -\frac{\partial Y_{\nu}\left(
z\right)  }{\partial\nu}J_{\nu+1}\left(  z\right)  }\right]  -\frac{J_{\nu
}\left(  z\right)  Y_{\nu}\left(  z\right)  }{2\nu}. \label{eq18}%
\end{equation}
We can obtain a similar expression by replacing $J$ by $Y$ in (\ref{eq16}),
noting that the value of $C_{0}$ may differ. In fact, from comparing both
sides of (\ref{eq16}) as $z\rightarrow\infty$ and referring to (\ref{eq5}) and
(\ref{eq6}) we find this time that $C_{0}=1/\left(  {2\nu}\right)  $.
Consequently, we arrive at
\begin{equation}
\int_{z}^{\infty}{\frac{Y_{\nu}^{2}\left(  t\right)  }{t}dt}=\frac{z}{2\nu
}\left[  {\frac{\partial Y_{\nu+1}\left(  z\right)  }{\partial\nu}Y_{\nu
}\left(  z\right)  -\frac{\partial Y_{\nu}\left(  z\right)  }{\partial\nu
}Y_{\nu+1}\left(  z\right)  }\right]  -\frac{Y_{\nu}^{2}\left(  z\right)
-1}{2\nu}. \label{eq19}%
\end{equation}
Finally, eliminating $\partial Y_{\nu+1}\left(  x\right)  /\partial\nu$ from
(\ref{eq18}) and (\ref{eq19}), and using (\ref{eq13}), yields (\ref{eq11})
\end{proof}

Using the behavior of Hankel functions at infinity, and in a similar manner to
the proof of (\ref{eq1}), one can show that the following result holds.
Details of the proof are left to the reader.

\begin{lemma}
\textit{For} $\nu>0$ \textit{and} $z\neq0$
\begin{equation}%
\begin{array}
[c]{l}%
\dfrac{\partial H_{\nu}^{\left(  1\right)  }\left(  z\right)  }{\partial\nu
}={\dfrac{1}{2}}\pi i\left[  {\nu H_{\nu}^{\left(  1\right)  }\left(
z\right)  \int_{z}^{\infty}{t^{-1}H_{\nu}^{\left(  1\right)  }\left(
t\right)  H_{\nu}^{\left(  2\right)  }\left(  t\right)  dt}}\right. \\
\left.  {-\nu H_{\nu}^{\left(  2\right)  }\left(  z\right)  \int_{z}^{\infty
}{t^{-1}\left\{  {H_{\nu}^{\left(  1\right)  }\left(  t\right)  }\right\}
^{2}dt}-H_{\nu}^{\left(  1\right)  }\left(  z\right)  }\right]  .
\end{array}
\label{eq20}%
\end{equation}
Likewise, for $\partial H_{\nu}^{\left(  2\right)  }\left(  z\right)
/\partial\nu$, in (\ref{eq20}) interchange the superscripts (1) and (2) and
replace $i$ by $-i$ throughout.
\end{lemma}

\section{Series expansions for integrals involving products of Bessel
functions}

Here we obtain a number of series expansions for the integrals appearing in
\S 2. While power series expansions exist for $\hat{{J}}_{\nu}\left(
z\right)  $ and $\hat{{Y}}_{\nu}\left(  z\right)  $ themselves (see [14,
\S 10.15], these are only useful for small $z$, or for special values of $\nu
$. Also, expansions of the integrals themselves is of importance in scattering
theory, as noted below.

Our results appear to be new, and are complementary to certain results
appearing in the literature. In particular, from [16, p. 152], we have the
following expansion.

\begin{lemma}
\textit{For positive x and }$\nu$
\begin{equation}
\int_{0}^{x}{\frac{J_{\nu}^{2}\left(  t\right)  }{t}dt}=\frac{1}{2\nu}%
\sum\limits_{n=0}^{\infty}{\varepsilon_{n}J_{\nu+n}^{2}\left(  x\right)  },
\label{eq21}%
\end{equation}
\textit{where }$\varepsilon_{0}=1$\textit{\ and }$\varepsilon_{n}%
=2$\textit{\ otherwise.}
\end{lemma}

This result was generalized in [6]; in that paper it is shown how series of
this type allows one to write the spherical harmonic component of the Coulomb
kernel $\left\vert {\mathrm{\mathbf{r}}-\mathrm{\mathbf{{r}^{\prime}}}%
}\right\vert ^{-1}$ as a numerically stable series.

The integral in (\ref{eq21}) diverges when $\nu=0$. However, from
(\ref{eq10}), we see that (\ref{eq21}) can be re-expressed in the form
\[
\int_{x}^{\infty}{\frac{J_{\nu}^{2}\left(  t\right)  }{t}dt}=\frac{1}{2\nu
}-\frac{1}{2\nu}\sum\limits_{n=0}^{\infty}{\varepsilon_{n}J_{\nu+n}^{2}\left(
x\right)  }.
\]
The integral here now converges when $\nu=0_{\,}$for each positive $x$, and
limiting value of the RHS can therefore be found as $\nu\rightarrow0$. In
particular, using L'Hopital's rule, we find that
\begin{equation}
\int_{x}^{\infty}{\frac{J_{0}^{2}\left(  t\right)  }{t}dt}=-\sum
\limits_{n=0}^{\infty}{\varepsilon_{n}J_{n}\left(  x\right)  \hat{{J}}%
_{n}\left(  x\right)  }. \label{eq23}%
\end{equation}
To obtain a similar result to (\ref{eq21}) for the integrals appearing in
(\ref{eq1}) and (\ref{eq11}), we shall require the following.

\begin{lemma}
\textit{For each positive fixed }$x$
\begin{equation}
\int_{x}^{\infty}{\frac{J_{\nu}\left(  t\right)  Y_{\nu}\left(  t\right)  }%
{t}dt}=\frac{1}{\nu\pi}\ln\left(  {\frac{x}{2\nu}}\right)  +O\left(  {\frac
{1}{\nu^{3}}}\right)  , \label{eq24}%
\end{equation}
\textit{as} $\nu\rightarrow\infty$.
\end{lemma}

\begin{proof}
Using
\begin{equation}
J_{\nu}\left(  x\right)  =\left(  {\frac{x}{2}}\right)  ^{\nu}\sum
\limits_{k=0}^{\infty}{\left(  {-1}\right)  ^{k}\frac{\left(  {{\frac{1}{4}%
}x^{2}}\right)  ^{k}}{k!\Gamma\left(  {\nu+k+1}\right)  }}, \label{eq25}%
\end{equation}
and, assuming temporarily that $\nu$ is not an integer,
\begin{equation}
Y_{\nu}\left(  x\right)  =\frac{J_{\nu}\left(  x\right)  \cos\left(  {\nu\pi
}\right)  -J_{-\nu}\left(  x\right)  }{\sin\left(  {\nu\pi}\right)  },
\label{eq26}%
\end{equation}
we find from (\ref{eq16}) that (\ref{eq24}) holds. By continuity the condition
$\nu$ not an integer can now be relaxed.
\end{proof}

\begin{proposition}
For positive x, and $\nu\neq0,-1,-2,-3,\cdots$
\begin{equation}
\int_{x}^{\infty}{\frac{J_{\nu}\left(  t\right)  Y_{\nu}\left(  t\right)  }%
{t}dt}=\frac{\ln\left(  {{\frac{1}{2}}x}\right)  -\psi\left(  \nu\right)
}{\nu\pi}-\frac{1}{2\nu}\sum\limits_{n=0}^{\infty}{\left\{  {\varepsilon
_{n}J_{\nu+n}\left(  x\right)  Y_{\nu+n}\left(  x\right)  +\frac{2}{\pi\left(
{\nu+n}\right)  }}\right\}  }, \label{eq27}%
\end{equation}
where $\psi\left(  \nu\right)  $\ is the Psi function [14, \S 5.2(i)]. For
$\nu=-1,-2,-3,\cdots$\ (\ref{eq27}) holds with $\nu$\ replaced by $\left\vert
\nu\right\vert $on the RHS.
\end{proposition}

\begin{proof}
Firstly we assume $\nu\neq0,-1,-2,-3,\cdots$. Then using
\[
\mathcal{C}_{p}^{\prime}\left(  t\right)  ={\frac{1}{2}}\mathcal{C}%
_{p-1}\left(  t\right)  -{\frac{1}{2}}\mathcal{C}_{p+1}\left(  t\right)  ,
\]
we have
\[%
\begin{array}
[c]{l}%
\left\{  {J_{\nu+n}\left(  t\right)  Y_{\nu+n}\left(  t\right)  }\right\}
^{\prime}={\frac{1}{2}}\left\{  {J_{\nu+n}\left(  t\right)  Y_{\nu+n-1}\left(
t\right)  -J_{\nu+n+1}\left(  t\right)  Y_{\nu+n}\left(  t\right)  }\right\}
\\
+{\frac{1}{2}}\left\{  {J_{\nu+n-1}\left(  t\right)  Y_{\nu+n}\left(
t\right)  -J_{\nu+n}\left(  t\right)  Y_{\nu+n+1}\left(  t\right)  }\right\}
.
\end{array}
\]
On summing both sides from $n=0$ to $n=N$ yields, via telescoping,
\begin{equation}%
\begin{array}
[c]{l}%
\sum\limits_{n=0}^{N}{\left\{  {J_{\nu+n}\left(  t\right)  Y_{\nu+n}\left(
t\right)  }\right\}  ^{\prime}}={\frac{1}{2}}J_{\nu}\left(  t\right)
Y_{\nu-1}\left(  t\right)  +{\frac{1}{2}}J_{\nu-1}\left(  t\right)  Y_{\nu
}\left(  t\right) \\
-{\frac{1}{2}}J_{\nu+N+1}\left(  t\right)  Y_{\nu+N}\left(  t\right)
-{\frac{1}{2}}J_{\nu+N}\left(  t\right)  Y_{\nu+N+1}\left(  t\right)  .
\end{array}
\label{eq30}%
\end{equation}
We plan to let $N\rightarrow\infty$. Next we use
\[
\mathcal{C}_{p-1}\left(  t\right)  =\left(  {p/t}\right)  \mathcal{C}%
_{p}\left(  t\right)  +\mathcal{C}_{p}^{\prime}\left(  t\right)  ,
\]
to obtain from (\ref{eq30})
\[%
\begin{array}
[c]{l}%
\sum\limits_{n=0}^{N}{\left\{  {J_{\nu+n}\left(  t\right)  Y_{\nu+n}\left(
t\right)  }\right\}  ^{\prime}}=\left(  {\nu/t}\right)  J_{\nu}\left(
t\right)  Y_{\nu}\left(  t\right)  +{\frac{1}{2}}\left\{  {J_{\nu}\left(
t\right)  Y_{\nu}\left(  t\right)  }\right\}  ^{\prime}\\
-\left(  {\left(  {\nu+N}\right)  /t}\right)  J_{\nu+N}\left(  t\right)
Y_{\nu+N}\left(  t\right)  -{\frac{1}{2}}\left\{  {J_{\nu+N}\left(  t\right)
Y_{\nu+N}\left(  t\right)  }\right\}  ^{\prime}.
\end{array}
\]
Integrating this from $t=x$ to $t=\infty$, and using (\ref{eq5}) and
(\ref{eq6}), then yields
\begin{equation}%
\begin{array}
[c]{l}%
{\frac{1}{2}}J_{\nu}\left(  x\right)  Y_{\nu}\left(  x\right)  -{\frac{1}{2}%
}J_{\nu+N}\left(  x\right)  Y_{\nu+N}\left(  x\right)  -\sum\limits_{n=0}%
^{N}{J_{\nu+n}\left(  x\right)  Y_{\nu+n}\left(  x\right)  }\\
=\nu\int_{x}^{\infty}{t^{-1}J_{\nu}\left(  t\right)  Y_{\nu}\left(  t\right)
dt}-\left(  {\nu+N}\right)  \int_{x}^{\infty}{t^{-1}J_{\nu+N}\left(  t\right)
Y_{\nu+N}\left(  t\right)  dt}.
\end{array}
\label{eq33}%
\end{equation}
But, for fixed $x>0$, we have from (\ref{eq25}) and (\ref{eq26})
\[
J_{\nu+n}\left(  x\right)  Y_{\nu+n}\left(  x\right)  =-\frac{1}{\pi\left(
{\nu+n}\right)  }-\frac{x^{2}}{2\pi\left(  {\nu+n}\right)  ^{3}}+O\left(
{\frac{1}{n^{4}}}\right)  ,
\]
as $n\rightarrow\infty$. Thus, in order to ensure that the series in
(\ref{eq33}) converges as $N\rightarrow\infty$, we rewrite the equation as
\begin{equation}%
\begin{array}
[c]{l}%
{\frac{1}{2}}J_{\nu}\left(  x\right)  Y_{\nu}\left(  x\right)  -{\frac{1}{2}%
}J_{\nu+N}\left(  x\right)  Y_{\nu+N}\left(  x\right)  -\sum\limits_{n=0}%
^{N}{\left\{  {J_{\nu+n}\left(  x\right)  Y_{\nu+n}\left(  x\right)
+\dfrac{1}{\pi\left(  {\nu+n}\right)  }}\right\}  }\\
=\nu\int_{x}^{\infty}{\dfrac{J_{\nu}\left(  t\right)  Y_{\nu}\left(  t\right)
}{t}dt}-\dfrac{1}{\pi}\sum\limits_{n=0}^{N}{\dfrac{1}{\nu+n}}-\left(  {\nu
+N}\right)  \int_{x}^{\infty}{\dfrac{J_{\nu+N}\left(  t\right)  Y_{\nu
+N}\left(  t\right)  }{t}dt}.
\end{array}
\label{eq35}%
\end{equation}
Next, it is well known that [14, Chap. 5]
\begin{equation}
\sum\limits_{n=0}^{N}{\frac{1}{\nu+n}}=\psi\left(  {\nu+N+1}\right)
-\psi\left(  \nu\right)  =\ln\left(  N\right)  -\psi\left(  \nu\right)
+O\left(  {\frac{1}{N}}\right)  \quad\left(  {N\rightarrow\infty}\right)  .
\label{eq36}%
\end{equation}
Moreover, from (\ref{eq24}),
\begin{equation}
\left(  {\nu+N}\right)  \int_{x}^{\infty}{\frac{J_{\nu+N}\left(  t\right)
Y_{\nu+N}\left(  t\right)  }{t}dt}=\frac{1}{\pi}\ln\left(  {\frac{x}{2N}%
}\right)  +O\left(  {\frac{1}{N}}\right)  . \label{eq37}%
\end{equation}
Using (\ref{eq36}) and (\ref{eq37}), and then letting $N\rightarrow\infty$ in
(\ref{eq35}), yields (\ref{eq27}).

Finally, for $\nu=-1,-2,-3,\cdots$, we observe that $J_{\nu}\left(  t\right)
Y_{\nu}\left(  t\right)  =J_{-\nu}\left(  t\right)  Y_{-\nu}\left(  t\right)
$ ([14, \S 10.4]), and hence (\ref{eq27}) holds with $\nu$ replaced by
$\left\vert \nu\right\vert $ on the RHS.
\end{proof}

In the case $\nu=0$ we take limits for the RHS of (\ref{eq27}), and we obtain
\begin{equation}
\int_{x}^{\infty}{\frac{J_{0}\left(  t\right)  Y_{0}\left(  t\right)  }{t}%
dt}=-\frac{1}{2}\sum\limits_{n=0}^{\infty}{\varepsilon_{n}\left.
{\frac{\partial}{\partial\nu}\left\{  {J_{\nu+n}\left(  x\right)  Y_{\nu
+n}\left(  x\right)  }\right\}  }\right\vert _{\nu=0}}. \label{eq38}%
\end{equation}
In deriving this we used L'Hopital's rule, along with the relation
\[
\frac{1}{\nu^{2}}-{\psi}^{\prime}\left(  \nu\right)  =-\frac{\pi^{2}}%
{6}+O\left(  \nu\right)  \quad\left(  {\nu\rightarrow0}\right)  .
\]
We note in passing that if we multiply both sides of (\ref{eq27}) by $\nu$,
and then set $\nu=0$, we obtain the following new result.

\begin{corollary}
\textit{For positive x}
\begin{equation}
\sum\limits_{n=1}^{\infty}{\left\{  {J_{n}\left(  x\right)  Y_{n}\left(
x\right)  +\frac{1}{\pi n}}\right\}  }=\frac{\ln\left(  {{\frac{1}{2}}%
x}\right)  +\gamma}{\pi}-\frac{1}{2}J_{0}\left(  x\right)  Y_{0}\left(
x\right)  . \label{eq40}%
\end{equation}

\end{corollary}

Finally, in order to obtain a complementary representation for $\int
_{x}^{\infty}{t^{-1}Y_{\nu}^{2}\left(  t\right)  dt}$, one can use the
identity
\begin{equation}
J_{-\nu}\left(  t\right)  Y_{-\nu}\left(  t\right)  =\cos\left(  {2\nu\pi
}\right)  J_{\nu}\left(  t\right)  Y_{\nu}\left(  t\right)  +{\frac{1}{2}}%
\sin\left(  {2\nu\pi}\right)  \left\{  {J_{\nu}^{2}\left(  t\right)  -Y_{\nu
}^{2}\left(  t\right)  }\right\}  . \label{eq41}%
\end{equation}
(see [14, (10.4.5)]. Thus, on integrating from $t=x$ to $t=\infty$, we obtain
\begin{equation}%
\begin{array}
[c]{l}%
\int_{x}^{\infty}{\dfrac{Y_{\nu}^{2}\left(  t\right)  }{t}dt}=\int_{x}%
^{\infty}{\dfrac{J_{\nu}^{2}\left(  t\right)  }{t}dt}+2\cot\left(  {2\nu\pi
}\right)  \int_{x}^{\infty}{\dfrac{J_{\nu}\left(  t\right)  Y_{\nu}\left(
t\right)  }{t}dt}\\
-2\csc\left(  {2\nu\pi}\right)  \int_{x}^{\infty}{\dfrac{J_{-\nu}\left(
t\right)  Y_{-\nu}\left(  t\right)  }{t}dt},
\end{array}
\label{eq42}%
\end{equation}
provided $2\nu$ is not an integer. Then one can substitute (\ref{eq10}),
(\ref{eq21}), and (\ref{eq27}) (with $\nu$ replaced by $-\nu$ for the third
integral in (\ref{eq42})) to obtain the desired result.

If $2\nu=p$ is an integer we can take limits in (\ref{eq41}), and in
particular use
\[%
\begin{array}
[c]{l}%
\lim\limits_{2\nu\rightarrow p}\dfrac{\cos\left(  {2\nu\pi}\right)  J_{\nu
}\left(  t\right)  Y_{\nu}\left(  t\right)  -J_{-\nu}\left(  t\right)
Y_{-\nu}\left(  t\right)  }{\sin\left(  {2\nu\pi}\right)  }\\
=\dfrac{1}{2\pi}\left.  {\dfrac{\partial}{\partial\nu}\left\{  {J_{\nu}\left(
t\right)  Y_{\nu}\left(  t\right)  }\right\}  }\right\vert _{\nu=p/2}+\left.
{\dfrac{\left(  {-1}\right)  ^{p}}{2\pi}\dfrac{\partial}{\partial\nu}\left\{
{J_{\nu}\left(  t\right)  Y_{\nu}\left(  t\right)  }\right\}  }\right\vert
_{\nu=-p/2}.
\end{array}
\]
As a result, we have
\begin{equation}%
\begin{array}
[c]{l}%
\int_{x}^{\infty}{\dfrac{Y_{p/2}^{2}\left(  t\right)  }{t}dt}=\int_{x}%
^{\infty}{\dfrac{J_{p/2}^{2}\left(  t\right)  }{t}dt}+\dfrac{1}{\pi}%
\dfrac{\partial}{\partial\nu}\left.  {\int_{x}^{\infty}{J_{\nu}\left(
t\right)  Y_{\nu}\left(  t\right)  dt}}\right\vert _{\nu=p/2}\\
+\dfrac{\left(  {-1}\right)  ^{p}}{\pi}\dfrac{\partial}{\partial\nu}\left.
{\int_{x}^{\infty}{J_{\nu}\left(  t\right)  Y_{\nu}\left(  t\right)  dt}%
}\right\vert _{\nu=-p/2}.
\end{array}
\label{eq44}%
\end{equation}

\section{Uniform asymptotic approximations for $\hat{{J}}_{\nu}\left(  {\nu
x}\right)  $ and $\hat{{Y}}_{\nu}\left(  {\nu x}\right)  $}

We now obtain uniform asymptotic approximations for the order derivatives of
$J_{\nu}\left(  x\right)  $ and $Y_{\nu}\left(  x\right)  $ for large $\nu$,
in terms of the Airy functions $\mathrm{Ai}\left(  x\right)  $ and
$\mathrm{Bi}\left(  x\right)  $. For brevity we restrict our consideration to
real argument and leading order approximations, although extensions to complex
argument and expansions are feasible from our methods. Explicit error bounds
are also available from the corresponding ones for the Bessel functions
themselves, but again for brevity we only include order estimates.

We note in passing the well-known behavior of Airy functions, given by
\begin{equation}
\mathrm{Ai}\left(  x\right)  \sim\frac{\exp\left(  {-{\frac{2}{3}}x^{3/2}%
}\right)  }{2\sqrt{\pi}x^{1/4}},\quad\mathrm{Bi}\left(  x\right)  \sim
\frac{\exp\left(  {{\frac{2}{3}}x^{3/2}}\right)  }{\sqrt{\pi}x^{1/4}}%
\quad\left(  {x\rightarrow\infty}\right)  , \label{eq45}%
\end{equation}
and
\begin{equation}
\mathrm{Ai}\left(  x\right)  \sim\frac{\cos\left(  {{\frac{2}{3}}\left\vert
x\right\vert ^{3/2}-{\frac{1}{4}}\pi}\right)  }{\sqrt{\pi}\left\vert
x\right\vert ^{1/4}},\quad\mathrm{Bi}\left(  x\right)  \sim-\frac{\sin\left(
{{\frac{2}{3}}\left\vert x\right\vert ^{3/2}-{\frac{1}{4}}\pi}\right)  }%
{\sqrt{\pi}\left\vert x\right\vert ^{1/4}}\quad\left(  {x\rightarrow-\infty
}\right)  . \label{eq46}%
\end{equation}
Moreover, $\mathrm{Ai}\left(  x\right)  $ and $\mathrm{Bi}\left(  x\right)  $
have no zeros for $0\leq x<\infty$.

We shall employ the corresponding uniform approximations for the Bessel
functions themselves, a brief description of which is given as follows. Let
$c=-0.36604\cdots$ be the largest real root of the equation $\mathrm{Ai}%
\left(  x\right)  =\mathrm{Bi}\left(  x\right)  $, and define the weight
function of $\mathrm{Ai}\left(  x\right)  $ and $\mathrm{Bi}\left(  x\right)
$ by
\[%
\begin{array}
[c]{l}%
E\left(  x\right)  =1\quad\left(  {-\infty<x\leq c}\right) \\
E\left(  x\right)  =\left\{  {\mathrm{Bi}\left(  x\right)  /\mathrm{Ai}\left(
x\right)  }\right\}  ^{1/2}\quad\left(  {c\leq x<\infty}\right)  ,
\end{array}
\]
and the modulus function by
\[%
\begin{array}
[c]{l}%
M\left(  x\right)  =\left\{  {\mathrm{Ai}^{2}\left(  x\right)  +\mathrm{Bi}%
^{2}\left(  x\right)  }\right\}  ^{1/2}\quad\left(  {-\infty<x\leq c}\right)
,\\
M\left(  x\right)  =\left\{  {2\mathrm{Ai}\left(  x\right)  \mathrm{Bi}\left(
x\right)  }\right\}  ^{1/2}\quad\left(  {c\leq x<\infty}\right)  .
\end{array}
\]
Next, introduce a new variable $\zeta$ by
\begin{equation}
\frac{2}{3}\zeta^{3/2}=\ln\left\{  {\frac{1+\left(  {1-x^{2}}\right)  ^{1/2}%
}{x}}\right\}  -\left(  {1-x^{2}}\right)  ^{1/2}\quad\left(  {0<x\leq
1}\right)  , \label{eq49}%
\end{equation}%
\begin{equation}
\frac{2}{3}\left(  {-\zeta}\right)  ^{3/2}=\left(  {x^{2}-1}\right)
^{1/2}-\sec^{-1}\left(  x\right)  \quad\left(  {1\leq x<\infty}\right)  ,
\label{eq50}%
\end{equation}
all functions taking their principal values, with $\zeta=\infty,0,-\infty$,
corresponding to $x=0,1,\infty$, respectively. Note $\zeta\rightarrow-\infty$
as $x\rightarrow\infty$ such that
\begin{equation}
x={\frac{2}{3}}\left(  {-\zeta}\right)  ^{3/2}+{\frac{1}{2}}\pi+O\left(
{\left\vert \zeta\right\vert ^{-3/2}}\right)  , \label{eq51}%
\end{equation}
and $\zeta\rightarrow\infty$ as $x\rightarrow0^{+}$ such that
\begin{equation}
x=2\exp\left(  {-{\frac{2}{3}}\zeta^{3/2}-1}\right)  +O\left\{  {\exp\left(
{-{\frac{4}{3}}\zeta^{3/2}}\right)  }\right\}  . \label{eq52}%
\end{equation}

For complex $x$ and $\zeta$, the transformation is given by (\ref{eq49}),
where the branches take their principal values when $x\in\left(  {0,1}\right)
$ and $\zeta\in\left(  {0,\infty}\right)  $, and are continuous elsewhere.

From [13, Chap. 11, \S 10] we then have
\begin{equation}
Y_{\nu}\left(  {\nu x}\right)  =-\frac{1}{\nu^{1/3}}\left(  {\frac{4\zeta
}{1-x^{2}}}\right)  ^{1/4}\left\{  {\mathrm{Bi}\left(  {\nu^{2/3}\zeta
}\right)  +\varepsilon_{1}\left(  {\nu,\zeta}\right)  }\right\}  ,
\label{eq53}%
\end{equation}
and
\begin{equation}
J_{\nu}\left(  {\nu x}\right)  =\frac{1}{\nu^{1/3}}\left(  {\frac{4\zeta
}{1-x^{2}}}\right)  ^{1/4}\left\{  {\mathrm{Ai}\left(  {\nu^{2/3}\zeta
}\right)  +\varepsilon_{2}\left(  {\nu,\zeta}\right)  }\right\}  ,
\label{eq54}%
\end{equation}
where
\begin{equation}
\varepsilon_{1}\left(  {\nu,\zeta}\right)  =O\left(  {\nu^{-1}}\right)
E\left(  {\nu^{2/3}\zeta}\right)  M\left(  {\nu^{2/3}\zeta}\right)  ,
\label{eq55}%
\end{equation}
and
\begin{equation}
\varepsilon_{2}\left(  {\nu,\zeta}\right)  =O\left(  {\nu^{-1}}\right)
M\left(  {\nu^{2/3}\zeta}\right)  /E\left(  {\nu^{2/3}\zeta}\right)  ,
\label{eq56}%
\end{equation}
uniformly for $0<x<\infty$ (i.e. $-\infty<\zeta<\infty)$.

We shall also utilize the uniform approximation
\begin{equation}
H_{\nu}^{\left(  1\right)  }\left(  {\nu x}\right)  =\frac{2e^{-\pi i/3}}%
{\nu^{1/3}}\left(  {\frac{4\zeta}{1-x^{2}}}\right)  ^{1/4}\mathrm{Ai}\left(
{e^{2\pi i/3}\nu^{2/3}\zeta}\right)  \left\{  {1+O\left(  {\frac{1}{\nu}%
}\right)  }\right\}  , \label{eq57}%
\end{equation}
where (for our purposes) $x$ and $\zeta$ are complex, with $-\pi\leq
\arg\left(  \zeta\right)  <-{\frac{1}{3}}\pi$, and correspondingly $x$ lying
in a certain unbounded subset of the upper half plane (see [13, Chap. 11,
\S 10]). A fundamental property of $\mathrm{Ai}\left(  {e^{2\pi i/3}\nu
^{2/3}\zeta}\right)  $ is that for large $\nu$ it is exponentially small in
the sector $-\pi<\arg\left(  \zeta\right)  <-{\frac{1}{3}}\pi$ (and likewise
$H_{\nu}^{\left(  1\right)  }\left(  {\nu x}\right)  $ for $x$ in the upper
half plane and lying outside a certain domain which contains the origin: see
Figs. 10.1 and 10.2 of [13, Chap. 11]).

Now, from (\ref{eq1}) and (\ref{eq11}) we have
\begin{equation}
\hat{{J}}_{\nu}\left(  {\nu x}\right)  =\nu\pi J_{\nu}\left(  {\nu x}\right)
I_{1}\left(  {\nu,\nu x}\right)  +\nu\pi Y_{\nu}\left(  {\nu x}\right)
I_{2}\left(  {\nu,\nu x}\right)  , \label{eq58}%
\end{equation}
and
\begin{equation}
\hat{{Y}}_{\nu}\left(  {\nu x}\right)  =\nu\pi J_{\nu}\left(  {\nu x}\right)
I_{3}\left(  {\nu,\nu x}\right)  -\nu\pi Y_{\nu}\left(  {\nu x}\right)
I_{1}\left(  {\nu,\nu x}\right)  -{\frac{1}{2}}\pi J_{\nu}\left(  {\nu
x}\right)  , \label{eq59}%
\end{equation}
where
\begin{equation}
I_{1}\left(  {\nu,t}\right)  =\int_{t}^{\infty}{\frac{J_{\nu}\left(  s\right)
Y_{\nu}\left(  s\right)  }{s}ds}, \label{eq60}%
\end{equation}%
\begin{equation}
I_{2}\left(  {\nu,t}\right)  =\int_{0}^{t}{\frac{J_{\nu}^{2}\left(  s\right)
}{s}ds}, \label{eq61}%
\end{equation}
and
\begin{equation}
I_{3}\left(  {\nu,t}\right)  =\int_{t}^{\infty}{\frac{Y_{\nu}^{2}\left(
s\right)  }{s}ds}. \label{eq62}%
\end{equation}
The Bessel functions in (\ref{eq58}) and (\ref{eq59}) can immediately be
replaced by their Airy function approximations, given above. Regarding the
integrals $I_{j}\left(  {\nu,t}\right)  $ ($j=1,2,3)$ we note, from
(\ref{eq53}) - (\ref{eq56}) and (\ref{eq46}), that the integrands in
(\ref{eq60}) and (\ref{eq62}) are highly oscillatory in the interval
$\nu<s<\infty$, and likewise for (\ref{eq61}) if $t>\nu$. This is detrimental
in numerical evaluations, so our goal is to approximate these integrals by
readily computable functions and/or integrals with monotonic or slowly varying integrands.

Our results are summarized in Theorem 4.2 below. In this section and the next,
$\zeta\left(  t\right)  $ appearing in an integral denotes the function given
by (\ref{eq49}) and (\ref{eq50}) with $x$ replaced by $t$, and similarly
$x\left(  \eta\right)  $ is given by (\ref{eq49}) and (\ref{eq50}) with
$\zeta$ replaced by $\eta.$

Before stating our main theorem, we shall require the following result.

\begin{lemma}
Let
\begin{equation}
L\left(  {\nu,t}\right)  =\int_{t}^{\infty}{\frac{\left\{  {H_{\nu}^{\left(
1\right)  }\left(  s\right)  }\right\}  ^{2}}{2s}ds}. \label{eq63}%
\end{equation}
Then as $\nu\rightarrow\infty$
\begin{equation}
L\left(  {\nu,\nu x}\right)  =\frac{4e^{-2\pi i/3}}{\nu^{2/3}}\int
_{x}^{i\infty}{\left(  {\frac{\zeta\left(  t\right)  }{1-t^{2}}}\right)
^{1/2}\frac{\mathrm{Ai}^{2}\left(  {e^{2\pi i/3}\nu^{2/3}\zeta\left(
t\right)  }\right)  }{t}dt}\left\{  {1+O\left(  {\frac{1}{\nu}}\right)
}\right\}  , \label{eq64}%
\end{equation}
uniformly for $1\leq x<\infty$. Here the path of integration lies in the first
quadrant and is bounded away from the boundary EPB in Fig. 10.1 of [13, Chap.
11, \S 10], except for the endpoint when it is located at $x=1$.
\end{lemma}

\begin{proof}
We make a simple change of variable $s=\nu t$ to obtain from (\ref{eq63})
\begin{equation}
L\left(  {\nu,\nu x}\right)  =\int_{x}^{\infty}{\frac{\left\{  {H_{\nu
}^{\left(  1\right)  }\left(  {\nu t}\right)  }\right\}  ^{2}}{2t}dt}.
\label{eq65}%
\end{equation}
We then use (\ref{eq57}) in the integrand, and deform the path of integration
from the positive $x$-axis to the stated path. The deformation is justifiable
since the integrand is holomorphic in the first quadrant, and exponentially
small at $x=i\infty$.
\end{proof}

Note that $e^{2\pi i/3}\zeta\rightarrow+\infty$ as $x\rightarrow i\infty$, and
hence the Airy function in the integrand of (\ref{eq64}) is non-zero and (in
absolute value) rapidly decreasing along the path of integration.

\begin{theorem}
For large $\nu$, uniformly for $0<x\leq1$ $(0\leq\zeta<\infty)$,%
\begin{equation}%
\begin{array}
[c]{l}%
I_{1}\left(  {\nu,\nu x}\right)  =\operatorname{Im}L\left(  {\nu,\nu}\right)
\\
-\dfrac{2}{\nu^{2/3}}\int_{x}^{1}{\left(  {\dfrac{\zeta\left(  t\right)
}{1-t^{2}}}\right)  ^{1/2}\dfrac{\mathrm{Ai}\left(  {\nu^{2/3}\zeta\left(
t\right)  }\right)  \mathrm{Bi}\left(  {\nu^{2/3}\zeta\left(  t\right)
}\right)  }{t}dt}\left\{  {1+O\left(  {\dfrac{1}{\nu}}\right)  }\right\}  ,
\end{array}
\label{eq66}%
\end{equation}
\begin{equation}
I_{2}\left(  {\nu,\nu x}\right)  =\frac{2}{\nu^{2/3}}\int_{0}^{x}{\left(
{\frac{\zeta\left(  t\right)  }{1-t^{2}}}\right)  ^{1/2}\frac{\mathrm{Ai}%
^{2}\left(  {\nu^{2/3}\zeta\left(  t\right)  }\right)  }{t}dt}\left\{
{1+O\left(  {\frac{1}{\nu}}\right)  }\right\}  , \label{eq67}%
\end{equation}
and
\begin{equation}%
\begin{array}
[c]{l}%
I_{3}\left(  {\nu,\nu x}\right)  =\dfrac{2}{\nu^{2/3}}\int_{x}^{1}{\left(
{\dfrac{\zeta\left(  t\right)  }{1-t^{2}}}\right)  ^{1/2}\dfrac{\mathrm{Bi}%
^{2}\left(  {\nu^{2/3}\zeta\left(  t\right)  }\right)  }{t}dt}\left\{
{1+O\left(  {\dfrac{1}{\nu}}\right)  }\right\}  -\operatorname{Re}L\left(
{\nu,\nu}\right) \\
+\dfrac{1}{\nu^{2/3}}\int_{1}^{\infty}{\left(  {\dfrac{\zeta\left(  t\right)
}{1-t^{2}}}\right)  ^{1/2}\dfrac{\mathrm{Ai}^{2}\left(  {\nu^{2/3}\zeta\left(
t\right)  }\right)  +\mathrm{Bi}^{2}\left(  {\nu^{2/3}\zeta\left(  t\right)
}\right)  }{t}dt}\left\{  {1+O\left(  {\dfrac{1}{\nu}}\right)  }\right\}  .
\end{array}
\label{eq68}%
\end{equation}
For large $\nu$, uniformly for $1\leq x<\infty$ ($-\infty<\zeta\leq0)$
\begin{equation}
I_{1}\left(  {\nu,\nu x}\right)  =\operatorname{Im}L\left(  {\nu,\nu
x}\right)  , \label{eq69}%
\end{equation}%
\begin{equation}%
\begin{array}
[c]{l}%
I_{2}\left(  {\nu,\nu x}\right)  =\dfrac{1}{2\nu}-\operatorname{Re}L\left(
{\nu,\nu x}\right) \\
-\dfrac{1}{\nu^{2/3}}\int_{x}^{\infty}{\left(  {\dfrac{\zeta\left(  t\right)
}{1-t^{2}}}\right)  ^{1/2}\dfrac{\mathrm{Ai}^{2}\left(  {\nu^{2/3}\zeta\left(
t\right)  }\right)  +\mathrm{Bi}^{2}\left(  {\nu^{2/3}\zeta\left(  t\right)
}\right)  }{t}dt}\left\{  {1+O\left(  {\dfrac{1}{\nu}}\right)  }\right\}  ,
\end{array}
\label{eq70}%
\end{equation}
and
\begin{equation}%
\begin{array}
[c]{l}%
I_{3}\left(  {\nu,\nu x}\right)  =-\operatorname{Re}L\left(  {\nu,\nu
x}\right) \\
+\dfrac{1}{\nu^{2/3}}\int_{x}^{\infty}{\left(  {\dfrac{\zeta\left(  t\right)
}{1-t^{2}}}\right)  ^{1/2}\dfrac{\mathrm{Ai}^{2}\left(  {\nu^{2/3}\zeta\left(
t\right)  }\right)  +\mathrm{Bi}^{2}\left(  {\nu^{2/3}\zeta\left(  t\right)
}\right)  }{t}dt}\left\{  {1+O\left(  {\dfrac{1}{\nu}}\right)  }\right\}  .
\end{array}
\label{eq71}%
\end{equation}
In these, $L\left(  {\nu,\nu x}\right)  $ is approximated by (\ref{eq64}).
\end{theorem}

\textbf{Remark}. The Airy functions in the integrals appearing in
(\ref{eq66}), (\ref{eq67}), and the first integral on the RHS of (\ref{eq68})
are all non-oscillatory, since their arguments are nonnegative. From
(\ref{eq46}) we also note that
\[
\mathrm{Ai}^{2}\left(  {\nu^{2/3}\zeta}\right)  +\mathrm{Bi}^{2}\left(
{\nu^{2/3}\zeta}\right)  \sim1/\left(  {\pi\nu^{1/3}\left\vert \zeta
\right\vert ^{1/2}}\right)  \quad\left(  {\nu^{2/3}\zeta\rightarrow-\infty
}\right)  ,
\]
and indeed, as seen (for example) from [14, Eq. (9.11.4)], this function is
monotonic. Therefore the second integral on the RHS of (\ref{eq68}), as well
as integrals appearing in (\ref{eq70}) and (\ref{eq71}) are slowly varying,
and consequently all the integrals in question are numerically stable.

\begin{proof}
Making the change of variable $s=\nu t$ in (\ref{eq60}) we have
\begin{equation}
I_{1}\left(  {\nu,\nu x}\right)  =\int_{x}^{\infty}{\frac{J_{\nu}\left(  {\nu
t}\right)  Y_{\nu}\left(  {\nu t}\right)  }{t}dt}, \label{eq73}%
\end{equation}
(with similar representations used for $I_{2}\left(  {\nu,\nu x}\right)  $ and
$I_{3}\left(  {\nu,\nu x}\right)  )$. We next use $H_{\nu}^{\left(  1\right)
}\left(  t\right)  =J_{\nu}\left(  t\right)  +iY_{\nu}\left(  t\right)  $ to
give the relation
\[
J_{\nu}\left(  t\right)  Y_{\nu}\left(  t\right)  ={\frac{1}{2}}%
\operatorname{Im}\left\{  {\left(  {H_{\nu}^{\left(  1\right)  }\left(
t\right)  }\right)  ^{2}}\right\}  ,
\]
for $t$ positive. Then, using this in (\ref{eq73}) and referring to
(\ref{eq65}), we see that (\ref{eq69}) follows.

For $0<x\leq1$ we express $I_{1}\left(  {\nu,\nu x}\right)  $ in the form
\begin{equation}
I_{1}\left(  {\nu,\nu x}\right)  =\int_{x}^{1}{\frac{J_{\nu}\left(  {\nu
t}\right)  Y_{\nu}\left(  {\nu t}\right)  }{t}dt}+\int_{1}^{\infty}%
{\frac{J_{\nu}\left(  {\nu t}\right)  Y_{\nu}\left(  {\nu t}\right)  }{t}dt}.
\label{eq75}%
\end{equation}
The first integral here has a desired monotonic integrand, and using
(\ref{eq53}) and (\ref{eq54}) it can be approximated by%
\begin{equation}%
\begin{array}
[c]{l}%
\int_{x}^{1}{\dfrac{J_{\nu}\left(  {\nu t}\right)  Y_{\nu}\left(  {\nu
t}\right)  }{t}dt}\\
=-\dfrac{2}{\nu^{2/3}}\int_{x}^{1}{\left(  {\dfrac{\zeta\left(  t\right)
}{1-t^{2}}}\right)  ^{1/2}\dfrac{\mathrm{Ai}\left(  {\nu^{2/3}\zeta\left(
t\right)  }\right)  \mathrm{Bi}\left(  {\nu^{2/3}\zeta\left(  t\right)
}\right)  }{t}dt}\left\{  {1+O\left(  {\dfrac{1}{\nu}}\right)  }\right\}  .
\end{array}
\label{eq76}%
\end{equation}
For the second integral, we have from (\ref{eq69}),
\begin{equation}
\int_{1}^{\infty}{\frac{J_{\nu}\left(  {\nu t}\right)  Y_{\nu}\left(  {\nu
t}\right)  }{t}dt}=I_{1}\left(  {\nu,\nu}\right)  =\operatorname{Im}L\left(
{\nu,\nu}\right)  . \label{eq77}%
\end{equation}
Thus (\ref{eq75}), (\ref{eq76}) and (\ref{eq77}) establishes (\ref{eq66}).

Similarly to (\ref{eq76}), we use (\ref{eq54}) in (\ref{eq61}) to obtain
(\ref{eq67}) for the case $0<x\leq1$. For $1\leq x<\infty$ we use (\ref{eq10})
to rewrite
\begin{equation}
I_{2}\left(  {\nu,\nu x}\right)  =\frac{1}{2\nu}-\int_{x}^{\infty}%
{\frac{J_{\nu}^{2}\left(  {\nu t}\right)  }{t}dt}. \label{eq78}%
\end{equation}
Next, in this we use
\[
J_{\nu}^{2}\left(  t\right)  ={\frac{1}{2}}\operatorname{Re}\left\{  {\left(
{H_{\nu}^{\left(  1\right)  }\left(  t\right)  }\right)  ^{2}}\right\}
+{\frac{1}{2}}H_{\nu}^{\left(  1\right)  }\left(  t\right)  H_{\nu}^{\left(
2\right)  }\left(  t\right)  ,
\]
to arrive at
\begin{equation}
\int_{x}^{\infty}{\frac{J_{\nu}^{2}\left(  {\nu t}\right)  }{t}dt}%
=\operatorname{Re}\int_{x}^{\infty}{\frac{\left\{  {H_{\nu}^{\left(  1\right)
}\left(  {\nu t}\right)  }\right\}  ^{2}}{2t}dt}+\int_{x}^{\infty}%
{\frac{J_{\nu}^{2}\left(  {\nu t}\right)  +Y_{\nu}^{2}\left(  {\nu t}\right)
}{2t}dt}. \label{eq80}%
\end{equation}
Thus (\ref{eq70}) follows from (\ref{eq53}), (\ref{eq54}), (\ref{eq63}),
(\ref{eq78}) and (\ref{eq80}).

Similarly to (\ref{eq80}), we use (\ref{eq62}) along with
\[
Y_{\nu}^{2}\left(  t\right)  ={\frac{1}{2}}H_{\nu}^{\left(  1\right)  }\left(
t\right)  H_{\nu}^{\left(  2\right)  }\left(  t\right)  -{\frac{1}{2}%
}\operatorname{Re}\left\{  {\left(  {H_{\nu}^{\left(  1\right)  }\left(
t\right)  }\right)  ^{2}}\right\}  ,
\]
to arrive at (\ref{eq71}).

Finally, for $0<x<1$, we write
\[
I_{3}\left(  {\nu,\nu x}\right)  =\int_{x}^{1}{\frac{Y_{\nu}^{2}\left(  {\nu
t}\right)  }{t}dt}+\int_{1}^{\infty}{\frac{Y_{\nu}^{2}\left(  {\nu t}\right)
}{t}dt}.
\]
The approximation (\ref{eq68}) then follows from using (\ref{eq53}) in the
first of these integrals, and using (\ref{eq71}) (with $x=1)$ for the second integral.
\end{proof}

\bigskip As an illustration, let $\nu=100$ and $x=0.5$. The exact values that
follow are obtained using the \textit{fdiff} routine in MAPLE with sufficient
precision. From (\ref{eq58}) we obtain the approximation $\hat{{J}}%
_{100}\left(  {50}\right)  \approx$ $-1.47735\times10^{-21}$, compared to the
exact value $\hat{{J}}_{100}\left(  {50}\right)  =$ $-1.47702...\times
10^{-21}$. Similarly, from (\ref{eq59}), we obtain the approximation $\hat
{{Y}}_{100}\left(  {50}\right)  \approx-4.31473\times10^{18}$, compared to the
value $\hat{{Y}}_{100}\left(  {50}\right)  =-4.31569\times10^{18}$.

Likewise, if we choose $\nu=100$ and $x=5$, we obtain the approximations
$\hat{{J}}_{100}\left(  {500}\right)  $\newline$\approx0.0150731$ and
$\hat{{Y}}_{100}\left(  {500}\right)  \approx-0.0470087$. In comparison, the
exact values are $\hat{{J}}_{100}\left(  {500}\right)  =0.0150695\cdots$ and
$\hat{{Y}}_{100}\left(  {500}\right)  =-0.0470099\cdots$.

\section{Asymptotic approximations of the integrals appearing in Theorem 4.2}

The integrals in Theorem 4.2 involve Airy functions, and as such it is
desirable to have explicit asymptotic approximations for them. These will
involve the following integrals
\begin{equation}
G_{i,j}^{k}\left(  {\nu,\zeta}\right)  =\int{\zeta^{k}y_{i}\left(  {\nu
^{2/3}\zeta}\right)  y_{j}\left(  {\nu^{2/3}\zeta}\right)  d\zeta},
\label{eq83}%
\end{equation}
for $i=1,2$, $j=1,2$, and $k=0,1,2$; here $y_{1}\left(  t\right)
=\mathrm{Ai}\left(  t\right)  $ and $y_{2}\left(  t\right)  =\mathrm{Bi}%
\left(  t\right)  $. The key is that these integrals can be explicitly
evaluated, and in particular from [2] we have (with $i$ and $j$ not
necessarily distinct)
\begin{equation}
G_{i,j}^{0}\left(  {\nu,\zeta}\right)  =\zeta y_{i}\left(  {\nu^{2/3}\zeta
}\right)  y_{j}\left(  {\nu^{2/3}\zeta}\right)  -\nu^{-2/3}y_{i}^{\prime
}\left(  {\nu^{2/3}\zeta}\right)  y_{j}^{\prime}\left(  {\nu^{2/3}\zeta
}\right)  , \label{eq84}%
\end{equation}%
\begin{equation}%
\begin{array}
[c]{l}%
G_{i,j}^{1}\left(  {\nu,\zeta}\right)  ={\frac{1}{3}}\zeta^{2}y_{i}\left(
{\nu^{2/3}\zeta}\right)  y_{j}\left(  {\nu^{2/3}\zeta}\right)  -{\frac{1}{3}%
}\nu^{-2/3}\zeta y_{i}^{\prime}\left(  {\nu^{2/3}\zeta}\right)  y_{j}^{\prime
}\left(  {\nu^{2/3}\zeta}\right) \\
+{\frac{1}{6}}\nu^{-4/3}\left\{  {y_{i}\left(  {\nu^{2/3}\zeta}\right)
y_{j}^{\prime}\left(  {\nu^{2/3}\zeta}\right)  +y_{i}^{\prime}\left(
{\nu^{2/3}\zeta}\right)  y_{j}\left(  {\nu^{2/3}\zeta}\right)  }\right\}  ,
\end{array}
\label{eq85}%
\end{equation}
and
\begin{equation}%
\begin{array}
[c]{l}%
G_{i,j}^{2}\left(  {\nu,\zeta}\right)  =-{\frac{1}{5}}\nu^{-2/3}\zeta^{2}%
y_{i}^{\prime}\left(  {\nu^{2/3}\zeta}\right)  y_{j}^{\prime}\left(
{\nu^{2/3}\zeta}\right) \\
+{\frac{1}{5}}\nu^{-4/3}\zeta\left\{  {y_{i}\left(  {\nu^{2/3}\zeta}\right)
y_{j}^{\prime}\left(  {\nu^{2/3}\zeta}\right)  +y_{i}^{\prime}\left(
{\nu^{2/3}\zeta}\right)  y_{j}\left(  {\nu^{2/3}\zeta}\right)  }\right\} \\
+{\frac{1}{5}}\left(  {\zeta^{3}-\nu^{-2}}\right)  y_{i}\left(  {\nu
^{2/3}\zeta}\right)  y_{j}\left(  {\nu^{2/3}\zeta}\right)  .
\end{array}
\label{eq86}%
\end{equation}
Next let
\begin{equation}
a_{0}=2^{-2/3},\quad a_{1}={\frac{2}{5}},\quad a_{2}={\frac{3}{{35}}}2^{2/3},
\label{eq87}%
\end{equation}
these being the first three coefficients in the Maclaurin expansion
\begin{equation}
\frac{\zeta}{1-x^{2}}=\sum\limits_{k=0}^{\infty}{a_{k}\zeta^{k}},\quad\left(
{\left\vert \zeta\right\vert <\left(  {{\frac{3}{2}}\pi}\right)  ^{2/3}%
}\right)  ; \label{eq88}%
\end{equation}
see [13, Chap. 11, \S 10]. Finally, let $x_{0}=1-\nu_{\,}^{-1/3}%
$(correspondingly $\zeta_{0}=\zeta\left(  {x_{0}}\right)  $ $=\left(  {2/\nu
}\right)  ^{1/3}+O\left(  {\nu^{-2/3}}\right)  )$, and $x_{1}=1+\nu
_{\,}^{-1/3} $(correspondingly $\zeta_{1}=\zeta\left(  {x_{1}}\right)
=-\left(  {2/\nu}\right)  ^{1/3}+O\left(  {\nu^{-2/3}}\right)  )$.

The desired results are stated as follows.

\begin{theorem}
For large $\nu$, uniformly for $0<x\leq1$ ($0\leq\zeta<\infty)$
\begin{equation}%
\begin{array}
[c]{l}%
\int_{x}^{1}{\left(  {\dfrac{\zeta\left(  t\right)  }{1-t^{2}}}\right)
^{1/2}\dfrac{\mathrm{Ai}\left(  {\nu^{2/3}\zeta\left(  t\right)  }\right)
\mathrm{Bi}\left(  {\nu^{2/3}\zeta\left(  t\right)  }\right)  }{t}dt}\\
=\left\{  {1+O\left(  {\dfrac{1}{\nu}}\right)  }\right\}  \sum\limits_{k=0}%
^{2}{a_{k}\left[  {G_{1,2}^{k}\left(  {\nu,\min\left\{  {\zeta,\zeta_{0}%
}\right\}  }\right)  -G_{1,2}^{k}\left(  {\nu,0}\right)  }\right]  }\\
+\dfrac{1}{2\pi\nu^{1/3}}\left\{  {1+O\left(  {\dfrac{1}{\nu}}\right)
}\right\}  \max\left\{  {\cosh^{-1}\left(  {\dfrac{1}{x}}\right)  -\cosh
^{-1}\left(  {\dfrac{1}{x_{0}}}\right)  ,0}\right\}  ,
\end{array}
\label{eq89}%
\end{equation}%
\begin{equation}
\int_{0}^{x}{\left(  {\frac{\zeta\left(  t\right)  }{1-t^{2}}}\right)
^{1/2}\frac{\mathrm{Ai}^{2}\left(  {\nu^{2/3}\zeta\left(  t\right)  }\right)
}{t}dt}=-\left(  {\frac{\zeta}{1-x^{2}}}\right)  G_{1,1}^{0}\left(  {\nu
,\zeta}\right)  \left\{  {1+O\left(  {\frac{1}{\nu}}\right)  }\right\}  ,
\label{eq90}%
\end{equation}
and%
\begin{equation}%
\begin{array}
[c]{l}%
\int_{x}^{1}{\left(  {\dfrac{\zeta\left(  t\right)  }{1-t^{2}}}\right)
^{1/2}\dfrac{\mathrm{Bi}^{2}\left(  {\nu^{2/3}\zeta\left(  t\right)  }\right)
}{t}dt}\\
=\left[  {\left(  {\dfrac{\zeta}{1-x^{2}}}\right)  G_{2,2}^{0}\left(
{\nu,\zeta}\right)  -a_{0}G_{2,2}^{0}\left(  {\nu,0}\right)  }\right]
\left\{  {1+O\left(  {\dfrac{1}{\nu}}\right)  }\right\}
\end{array}
\label{eq91a}%
\end{equation}
In (\ref{eq89}) the nonnegative branch of the inverse hyperbolic cosine is taken.

For large $\nu$, uniformly for $1\leq x<\infty$ ($-\infty<\zeta\leq0)$
\begin{equation}
L\left(  {\nu,\nu x}\right)  =\frac{4e^{-2\pi i/3}}{\nu^{2/3}}\left(
{\frac{\zeta}{1-x^{2}}}\right)  G_{1,1}^{0}\left(  {e^{\pi i/3}\nu,\zeta
}\right)  \left\{  {1+O\left(  {\frac{1}{\nu}}\right)  }\right\}  ,
\label{eq92}%
\end{equation}
and
\begin{equation}%
\begin{array}
[c]{l}%
\int_{x}^{\infty}{\left(  {\dfrac{\zeta\left(  t\right)  }{1-t^{2}}}\right)
^{1/2}\dfrac{\mathrm{Ai}^{2}\left(  {\nu^{2/3}\zeta\left(  t\right)  }\right)
+\mathrm{Bi}^{2}\left(  {\nu^{2/3}\zeta\left(  t\right)  }\right)  }{t}dt}\\
=\dfrac{1}{\pi\nu^{1/3}}\left\{  {1+O\left(  {\dfrac{1}{\nu}}\right)
}\right\}  \left[  \sin^{-1}\left(  {\dfrac{1}{\max\left\{  {x,x_{1}}\right\}
}}\right)  \right. \\
\left.  +\max\left\{  {\sum\limits_{k=0}^{2}{a_{k}\left[  {G_{1,1}^{k}\left(
{\nu,\zeta}\right)  +G_{2,2}^{k}\left(  {\nu,\zeta}\right)  -G_{1,1}%
^{k}\left(  {\nu,\zeta_{1}}\right)  -G_{2,2}^{k}\left(  {\nu,\zeta_{1}%
}\right)  }\right]  },0}\right\}  \right]
\end{array}
\label{eq93}%
\end{equation}
with the principal inverse sine taken in (\ref{eq93}).
\end{theorem}

\begin{proof}
\bigskip\ For the integrals (\ref{eq90}), (\ref{eq91a}), and (\ref{eq92}) that
have rapidly changing integrands (i.e. exponentially decreasing or
increasing), we use integration by parts. We consider the latter, with other
two results being similarly proven. Applying integration by parts to the
integral appearing in (\ref{eq64}) we obtain
\begin{equation}%
\begin{array}
[c]{l}%
\int_{\infty e^{-2\pi i/3}}^{\zeta}{\left(  {\dfrac{\eta}{1-x\left(
\eta\right)  ^{2}}}\right)  \mathrm{Ai}^{2}\left(  {e^{2\pi i/3}\nu^{2/3}\eta
}\right)  d\eta}=\left(  {\dfrac{\zeta}{1-x^{2}}}\right)  G_{1,1}^{0}\left(
{e^{\pi i/3}\nu,\zeta}\right) \\
-\int_{\infty e^{-2\pi i/3}}^{\zeta}{\left(  {\dfrac{\eta}{1-x\left(
\eta\right)  ^{2}}}\right)  ^{\prime}G_{1,1}^{0}\left(  {e^{\pi i/3}\nu,\eta
}\right)  d\eta}.
\end{array}
\label{eq94}%
\end{equation}
Here the prime represents differentiation with respect to $\eta$, and from
(\ref{eq84}) we have
\begin{equation}
G_{1,1}^{0}\left(  {e^{\pi i/3}\nu,\zeta}\right)  =\zeta\mathrm{Ai}^{2}\left(
{e^{2\pi i/3}\nu^{2/3}\zeta}\right)  -e^{-2\pi i/3}\nu^{-2/3}\mathrm{Ai}%
^{\prime2}\left(  {e^{2\pi i/3}\nu^{2/3}\zeta}\right)  . \label{eq95}%
\end{equation}
Now the first of (\ref{eq45}) holds for complex $x$ with $\left\vert
{\arg\left(  x\right)  }\right\vert \leq\pi-\delta$ ($\delta>0)$ and the
principal value of $x^{3/2}$ taken. Then, from that equation, along with
(\ref{eq83}) and (\ref{eq95}), we can readily show that $G_{1,1}^{0}\left(
{e^{\pi i/3}\nu,\zeta}\right)  $ is non-vanishing for nonnegative $\zeta$,
with
\[
G_{1,1}^{0}\left(  {e^{\pi i/3}\nu,\zeta}\right)  =-\frac{e^{-2\pi i/3}%
3^{1/3}\Gamma^{2}\left(  {{\frac{2}{3}}}\right)  }{4\pi^{2}\nu^{2/3}}+O\left(
\zeta\right)  ,
\]
as $\zeta\rightarrow0$, and
\[
G_{1,1}^{0}\left(  {e^{\pi i/3}\nu,\zeta}\right)  =-\frac{e^{-4\pi i/3}%
\exp\left\{  {{\frac{4}{3}}\nu\zeta^{3/2}}\right\}  }{8\pi\nu^{4/3}\zeta
}\left\{  {1+O\left(  {\frac{1}{\nu^{1/3}\zeta^{3/2}}}\right)  }\right\}  ,
\]
as $\nu\zeta^{3/2}\rightarrow\infty$ such that $\left\vert {\arg\left(
{e^{2\pi i/3}\zeta}\right)  }\right\vert \leq\pi-\delta$. Hence we can assert
that
\begin{equation}
K_{1}\frac{\left\vert {\exp\left\{  {{\frac{4}{3}}\nu\zeta^{3/2}}\right\}
}\right\vert }{\nu^{4/3}\left\vert \zeta\right\vert +\nu^{2/3}}\leq\left\vert
{G_{1,1}^{0}\left(  {e^{\pi i/3}\nu,\zeta}\right)  }\right\vert \leq
K_{2}\frac{\left\vert {\exp\left\{  {{\frac{4}{3}}\nu\zeta^{3/2}}\right\}
}\right\vert }{\nu^{4/3}\left\vert \zeta\right\vert +\nu^{2/3}}, \label{eq98}%
\end{equation}
for some positive constants $K_{1}$ and $K_{2}$. Next, from the complex
variable extension of (\ref{eq49}) it is straightforward to show that
\begin{equation}
\left(  {\frac{\zeta}{1-x^{2}}}\right)  =O\left(  {\frac{1}{\zeta^{2}}%
}\right)  ,\quad\frac{d}{d\zeta}\left(  {\frac{\zeta}{1-x^{2}}}\right)
=O\left(  {\frac{1}{\zeta^{3}}}\right)  , \label{eq99}%
\end{equation}
as $\zeta\rightarrow\infty$ in the sector $\left\vert {\arg\left(  {-\zeta
}\right)  }\right\vert <2\pi/3$. Therefore from this (with $x$, $\zeta$
replaced by $x\left(  \eta\right)  $, $\eta$, respectively), and the upper
bound in (\ref{eq98}), we have for large $\nu$%
\[%
\begin{array}
[c]{l}%
\int_{\infty e^{-2\pi i/3}}^{\zeta}{\left(  {\dfrac{\eta}{1-x\left(
\eta\right)  ^{2}}}\right)  ^{\prime}G_{1,1}^{0}\left(  {e^{\pi i/3}\nu,\eta
}\right)  d\eta}\\
=O\left(  1\right)  \int_{\infty e^{-2\pi i/3}}^{\zeta}{\dfrac{\left\vert
{\exp\left\{  {{\frac{4}{3}}\nu\eta^{3/2}}\right\}  }\right\vert }{\left(
{\nu^{4/3}\left\vert \eta\right\vert +\nu^{2/3}}\right)  \left(  {\left\vert
\eta\right\vert +1}\right)  ^{3}}d\left\vert \eta\right\vert },
\end{array}
\]
uniformly for $-\infty<\zeta\leq0$. Now express $\eta=re^{-i\theta\left(
r\right)  }$ on the contour, where $r=\left\vert \eta\right\vert $, and
$\theta\left(  r\right)  $ decreases continuously from $\pi$ to $2\pi/3$ as
$r$ increases from $\left\vert \zeta\right\vert $ to $\infty$. From Laplace's
method ([13, Chap. 3, \S 7]), we arrive at
\begin{equation}%
\begin{array}
[c]{l}%
\int_{\infty e^{-2\pi i/3}}^{\zeta}{\left(  {\dfrac{\eta}{1-x\left(
\eta\right)  ^{2}}}\right)  ^{\prime}G_{1,1}^{0}\left(  {e^{\pi i/3}\nu,\eta
}\right)  d\eta}\\
=O\left(  1\right)  \int_{\left\vert \zeta\right\vert }^{\infty}{\dfrac
{\exp\left\{  {{\frac{4}{3}}\nu r^{3/2}\cos\left(  {{\frac{3}{2}}\theta\left(
r\right)  }\right)  }\right\}  }{\left(  {\nu^{4/3}r+\nu^{2/3}}\right)
\left(  {r+1}\right)  ^{3}}dr}\\
=O\left(  {\dfrac{1}{\nu^{5/3}\left(  {\nu^{2/3}\left\vert \zeta\right\vert
+1}\right)  \left(  {\left\vert \zeta\right\vert +1}\right)  ^{3}}}\right)  .
\end{array}
\label{eq101}%
\end{equation}
The result (\ref{eq92}) follows from (\ref{eq64}), (\ref{eq94}), (\ref{eq98}),
(\ref{eq99}) and (\ref{eq101}). The proofs of (\ref{eq90}) and (\ref{eq91a})
are similar.

Consider next the integrals with slowly varying (non-exponential) integrands.
The method of integration by parts is not effective in approximating these,
and instead we resort to a combination of a Maclaurin series expansion for
small values of the integration variable, coupled with asymptotic
approximations of Airy functions for the other unbounded values. To this end,
on using
\[
\frac{d\zeta}{dx}=-\frac{1}{x}\left(  {\frac{1-x^{2}}{\zeta}}\right)  ^{1/2},
\]
we express (\ref{eq89}) in the form
\begin{equation}%
\begin{array}
[c]{l}%
\int_{x}^{1}{\left(  {\dfrac{\zeta\left(  t\right)  }{1-t^{2}}}\right)
^{1/2}\dfrac{\mathrm{Ai}\left(  {\nu^{2/3}\zeta\left(  t\right)  }\right)
\mathrm{Bi}\left(  {\nu^{2/3}\zeta\left(  t\right)  }\right)  }{t}dt}\\
=\int_{0}^{\zeta_{0}}{\left(  {\dfrac{\eta}{1-x\left(  \eta\right)  ^{2}}%
}\right)  \mathrm{Ai}\left(  {\nu^{2/3}\eta}\right)  \mathrm{Bi}\left(
{\nu^{2/3}\eta}\right)  d\eta}\\
+\int_{x}^{x_{0}}{\left(  {\dfrac{\zeta\left(  t\right)  }{1-t^{2}}}\right)
^{1/2}\dfrac{\mathrm{Ai}\left(  {\nu^{2/3}\zeta\left(  t\right)  }\right)
\mathrm{Bi}\left(  {\nu^{2/3}\zeta\left(  t\right)  }\right)  }{t}dt}.
\end{array}
\label{eq103}%
\end{equation}
If $0\leq\zeta\leq\zeta_{0}$ ($x_{0}\leq x\leq1)$ the upper limit in the first
integral on the RHS is replaced by $\zeta$ and the second integral is null.

For the first integral on the RHS of (\ref{eq103}) we find from (\ref{eq88}),
and recalling $\zeta_{0}=\left(  {2/\nu}\right)  ^{1/3}+O\left(  {\nu^{-2/3}%
}\right)  $, that
\begin{equation}%
\begin{array}
[c]{l}%
\int_{0}^{\zeta_{0}}{\left(  {\dfrac{\eta}{1-x\left(  \eta\right)  ^{2}}%
}\right)  \mathrm{Ai}\left(  {\nu^{2/3}\eta}\right)  \mathrm{Bi}\left(
{\nu^{2/3}\eta}\right)  d\eta}\\
=\left\{  {1+O\left(  {\dfrac{1}{\nu}}\right)  }\right\}  \int_{0}^{\zeta_{0}%
}{\left\{  {a_{0}+a_{1}\eta+a_{2}\eta^{2}}\right\}  \mathrm{Ai}\left(
{\nu^{2/3}\eta}\right)  \mathrm{Bi}\left(  {\nu^{2/3}\eta}\right)  d\eta}.
\end{array}
\label{eq104}%
\end{equation}
Explicit integration and referring to (\ref{eq83}) yields the first term on
the RHS of (\ref{eq89}). Clearly this also holds if $\zeta_{0}$ is replaced by
$\zeta$ when $0\leq\zeta\leq\zeta_{0}$.

Next, from (\ref{eq45})
\[
\mathrm{Ai}\left(  {\nu^{2/3}\zeta}\right)  \mathrm{Bi}\left(  {\nu^{2/3}%
\zeta}\right)  =\frac{1}{2\pi\nu^{1/3}\zeta^{1/2}}\left\{  {1+O\left(
{\frac{1}{\nu^{2}\zeta^{3}}}\right)  }\right\}  ,
\]
as $\nu^{2/3}\zeta\rightarrow\infty$. Hence, assuming $0<x<x_{0}$ ($\zeta
_{0}<\zeta<\infty)$, we have $\zeta^{-1}=O\left(  {\nu^{1/3}}\right)  $ in the
second integral on the RHS of (\ref{eq103}), and so it follows that%
\begin{equation}%
\begin{array}
[c]{l}%
\int_{x}^{x_{0}}{\left(  {\dfrac{\zeta\left(  t\right)  }{1-t^{2}}}\right)
^{1/2}\dfrac{\mathrm{Ai}\left(  {\nu^{2/3}\zeta\left(  t\right)  }\right)
\mathrm{Bi}\left(  {\nu^{2/3}\zeta\left(  t\right)  }\right)  }{t}dt}\\
=\dfrac{1}{2\pi\nu^{1/3}}\left\{  {1+O\left(  {\dfrac{1}{\nu}}\right)
}\right\}  \int_{x}^{x_{0}}{\left(  {\dfrac{1}{1-t^{2}}}\right)  ^{1/2}%
\dfrac{1}{t}dt}.
\end{array}
\label{eq106a}%
\end{equation}
\newline Now for $0<t\leq1$
\begin{equation}
\int{\left(  {\frac{1}{1-t^{2}}}\right)  ^{1/2}\frac{1}{t}dt}=-\cosh
^{-1}\left(  {\frac{1}{t}}\right)  . \label{eq107}%
\end{equation}
On combining (\ref{eq106a}) and (\ref{eq107}) we then arrive at the second
term on the RHS of (\ref{eq89}).(\ref{eq106a})

The proof of (\ref{eq93}) is similar. For $1\leq x<x_{1}$ ($\zeta_{1}%
<\zeta\leq0)$ we write this as
\begin{equation}%
\begin{array}
[c]{l}%
\int_{x}^{\infty}{\left(  {\dfrac{\zeta\left(  t\right)  }{1-t^{2}}}\right)
^{1/2}\dfrac{\mathrm{Ai}^{2}\left(  {\nu^{2/3}\zeta\left(  t\right)  }\right)
+\mathrm{Bi}^{2}\left(  {\nu^{2/3}\zeta\left(  t\right)  }\right)  }{t}dt}\\
=\int_{x_{1}}^{\infty}{\left(  {\dfrac{\zeta\left(  t\right)  }{1-t^{2}}%
}\right)  ^{1/2}\dfrac{\mathrm{Ai}^{2}\left(  {\nu^{2/3}\zeta\left(  t\right)
}\right)  +\mathrm{Bi}^{2}\left(  {\nu^{2/3}\zeta\left(  t\right)  }\right)
}{t}dt}\\
+\int_{\zeta_{1}}^{\zeta}{\left(  {\dfrac{\eta}{1-x\left(  \eta\right)  ^{2}}%
}\right)  \left\{  {\mathrm{Ai}^{2}\left(  {\nu^{2/3}\eta}\right)
+\mathrm{Bi}^{2}\left(  {\nu^{2/3}\eta}\right)  }\right\}  d\eta},
\end{array}
\label{eq108}%
\end{equation}
with no such splitting being necessary if $-\infty<\zeta\leq\zeta_{1}$
($x_{1}\leq x<\infty)$. Now using
\[
\mathrm{Ai}^{2}\left(  {\nu^{2/3}\zeta}\right)  +\mathrm{Bi}^{2}\left(
{\nu^{2/3}\zeta}\right)  =\frac{1}{\pi\nu^{1/3}\left(  {-\zeta}\right)
^{1/2}}\left\{  {1+O\left(  {\frac{1}{\nu^{2}\zeta^{3}}}\right)  }\right\}
\quad\left(  {\nu^{2/3}\zeta\rightarrow-\infty}\right)  ,
\]
we find for the first integral on the RHS of (\ref{eq108}) that%
\[%
\begin{array}
[c]{l}%
\int_{x_{1}}^{\infty}{\left(  {\dfrac{\zeta\left(  t\right)  }{1-t^{2}}%
}\right)  ^{1/2}\dfrac{\mathrm{Ai}^{2}\left(  {\nu^{2/3}\zeta\left(  t\right)
}\right)  +\mathrm{Bi}^{2}\left(  {\nu^{2/3}\zeta\left(  t\right)  }\right)
}{t}dt}\\
=\dfrac{1}{\pi\nu^{1/3}}\left\{  {1+O\left(  {\dfrac{1}{\nu}}\right)
}\right\}  \int_{x_{1}}^{\infty}{\left(  {\dfrac{1}{t^{2}-1}}\right)
^{1/2}\dfrac{1}{t}dt}.
\end{array}
\]
Then with
\[
\int_{x_{1}}^{\infty}{\left(  {\frac{1}{t^{2}-1}}\right)  ^{1/2}\frac{1}{t}%
dt}=\sin^{-1}\left(  {\frac{1}{x_{1}}}\right)  ,
\]
we arrive at the first term on the RHS of (\ref{eq93}).

In the case when splitting is necessary (i.e. $1\leq x<x_{1})$, we follow
(\ref{eq104}) and employ the Maclaurin expansion (\ref{eq88}) for the second
integral on the RHS of (\ref{eq108}), and the second term on the RHS of
(\ref{eq93}) is obtained.
\end{proof}

\bigskip Table 1 illustrates the results (\ref{eq89}), (\ref{eq90}),
(\ref{eq91a}) and (\ref{eq93}) numerically for the case $\nu=50$ and various
values of $x$. In this, we denote the relative errors $\eta_{l}\left(  {\nu
,x}\right)  $ ($l=1,2,3,4)$ by
\[
f_{l}\left(  {\nu,x}\right)  =g_{l}\left(  {\nu,x}\right)  \left\{
{1+\eta_{l}\left(  {\nu,x}\right)  }\right\}  ,
\]
where
\[
f_{1}\left(  {\nu,x}\right)  =\int_{x}^{1}{\left(  {\frac{\zeta\left(
t\right)  }{1-t^{2}}}\right)  ^{1/2}\frac{\mathrm{Ai}\left(  {\nu^{2/3}%
\zeta\left(  t\right)  }\right)  \mathrm{Bi}\left(  {\nu^{2/3}\zeta\left(
t\right)  }\right)  }{t}dt},
\]%
\[%
\begin{array}
[c]{l}%
g_{1}\left(  {\nu,x}\right)  =\sum\limits_{k=0}^{2}{a_{k}\left[  {G_{1,2}%
^{k}\left(  {\nu,\min\left\{  {\zeta,\zeta_{0}}\right\}  }\right)
-G_{1,2}^{k}\left(  {\nu,0}\right)  }\right]  }\\
+\dfrac{1}{2\pi\nu^{1/3}}\max\left\{  {\cosh^{-1}\left(  {\dfrac{1}{x}%
}\right)  -\cosh^{-1}\left(  {\dfrac{1}{x_{0}}}\right)  ,0}\right\}  ,
\end{array}
\]%
\[
f_{2}\left(  {\nu,x}\right)  =\int_{0}^{x}{\left(  {\frac{\zeta\left(
t\right)  }{1-t^{2}}}\right)  ^{1/2}\frac{\mathrm{Ai}^{2}\left(  {\nu
^{2/3}\zeta\left(  t\right)  }\right)  }{t}dt},
\]%
\[
g_{2}\left(  {\nu,x}\right)  =-\left(  {\frac{\zeta}{1-x^{2}}}\right)
G_{1,1}^{0}\left(  {\nu,\zeta}\right)  ,
\]%
\[
f_{3}\left(  {\nu,x}\right)  =\int_{x}^{1}{\left(  {\frac{\zeta\left(
t\right)  }{1-t^{2}}}\right)  ^{1/2}\frac{\mathrm{Bi}^{2}\left(  {\nu
^{2/3}\zeta\left(  t\right)  }\right)  }{t}dt},
\]%
\[
g_{3}\left(  {\nu,x}\right)  =\left[  {\left(  {\frac{\zeta}{1-x^{2}}}\right)
G_{2,2}^{0}\left(  {\nu,\zeta}\right)  -a_{0}G_{2,2}^{0}\left(  {\nu
,0}\right)  }\right]  ,
\]%
\[
f_{4}\left(  {\nu,x}\right)  =\int_{x}^{\infty}{\left(  {\frac{\zeta\left(
t\right)  }{1-t^{2}}}\right)  ^{1/2}\frac{\mathrm{Ai}^{2}\left(  {\nu
^{2/3}\zeta\left(  t\right)  }\right)  +\mathrm{Bi}^{2}\left(  {\nu^{2/3}%
\zeta\left(  t\right)  }\right)  }{t}dt},
\]
and
\[%
\begin{array}
[c]{l}%
g_{4}\left(  {\nu,x}\right)  =\dfrac{1}{\pi\nu^{1/3}}\sin^{-1}\left(
{\dfrac{1}{\max\left\{  {x,x_{1}}\right\}  }}\right) \\
+\max\left\{  {\sum\limits_{k=0}^{2}{a_{k}\left[  {G_{1,1}^{k}\left(
{\nu,\zeta}\right)  +G_{2,2}^{k}\left(  {\nu,\zeta}\right)  -G_{1,1}%
^{k}\left(  {\nu,\zeta_{1}}\right)  -G_{2,2}^{k}\left(  {\nu,\zeta_{1}%
}\right)  }\right]  },0}\right\}  .
\end{array}
\]
All calculations were performed with MAPLE, with the integrals $f_{l}\left(
{\nu,x}\right)  $ being evaluated numerically using Simpson's method.%

\begin{tabular}
[c]{ccccc}%
$x$ & $\left\vert {\eta_{1}\left(  {50,x}\right)  }\right\vert $ & $\left\vert
{\eta_{2}\left(  {50,x}\right)  }\right\vert $ & $\left\vert {\eta_{3}\left(
{50,x}\right)  }\right\vert $ & $\left\vert {\eta_{4}\left(  {50,x}\right)
}\right\vert $\\
$0.1$ & $1.6240E-04$ & $3.1202E-03$ & $3.1466E-03$ & -\\
$0.5$ & $3.4440E-04$ & $6.9778E-03$ & $7.2109E-03$ & -\\
$0.75$ & $2.1710E-04$ & $1.0326E-02$ & $1.1859E-02$ & -\\
$0.99$ & $1.7825E-08$ & $2.0756E-02$ & $1.8467E-02$ & -\\
$1$ & - & - & - & $2.6086E-04$\\
$5$ & - & - & - & $6.2709E-07$\\
$10$ & - & - & - & $1.1871E-07$%
\end{tabular}

\bigskip

\textit{Table 1.}

\textbf{Acknowledgement}. I thank the referees for a number of helpful
suggestions and comments.


\begin{thebibliography}{99}                                                                                               %


\bibitem {1}M. Abramowitz and I. A. Stegun, eds., \textit{Handbook of
Mathematical Functions with Formulas, Graphs, and Mathematical Tables}.
National Bureau of Standards Applied Mathematics Series, U.S. Government
Printing Office, Washington, D.C., 1964.

\bibitem {2}J. R. Albright, Integrals of products of Airy functions. J. Phys.
A 10 (4) (1977), 485-490.

\bibitem {3}A. Apelblat and N. Kravitsky, Integral Representations of
Derivatives and Integrals with Respect to the Order of the Bessel Functions
$J_{\nu}\left(  t\right)  $, $I_{\nu}\left(  t\right)  $, the Anger Function
$\mathrm{\mathbf{J}}_{\nu}\left(  t\right)  $, and the Integral Bessel
Function. $Ji_{\nu}\left(  t\right)  $. IMA J. Appl. Math. 34 (2) (1985), 187-210.

\bibitem {4}Yu. A. Brychkov and K. O. Geddes, On the derivatives of the Bessel
and Struve functions with respect to the order. Integral Transforms Spec.
Funct. 16 (2005), 187-198.

\bibitem {5}H. S. Cohl, Derivatives with respect to the degree and order of
associated Legendre functions for
$\vert$%
z%
$\vert$%
$>$%
1 using modified Bessel functions. Integral Transforms. Spec. Funct. 21
(2010), no. 7-8, 581-588

\bibitem {6}D. E. Dominici, and P. M. W. Gill, A remarkable identity involving
Bessel functions. Proc. R. Soc. Lond. Ser. A Math. Phys. Eng. Sci. 468 (2012), 2667-2681.

\bibitem {7}A. Erd\'{e}lyi, W. Magnus, F. Oberhettinger and F. G. Tricomi,
\textit{Higher Transcendental Functions}. Vol. II. McGraw-Hill Book Company,
Inc., New York-Toronto-London, 1953.

\bibitem {8}I. S. Gradsteyn and I. M. Rhyzik. \textit{Table of integrals,
series and products}. Academic Press, 2006.

\bibitem {9}L. J. Landau, Bessel functions: monotonicity and bounds. J. London
Math. Soc. (2) 61 (2000), 197-215.

\bibitem {10}Y. D. Luke, \textit{The special functions and their
approximations}, vol. II. Academic Press, 1969.

\bibitem {11}W. Magnus, F. Oberhettinger and R. P. Soni, \textit{Formulas and
Theorems for the Special Functions of Mathematical Physics}. 3rd edition,
Springer-Verlag, New York-Berlin, 1966.

\bibitem {12}F. Oberhettinger, On the derivative of Bessel functions with
respect to the order. J. Math. and Phys. 37 (1958), 75-78.

\bibitem {13}F. W. J. Olver, \textit{Asymptotics and special functions}.
Reprint of the 1974 original [Academic Press, New York]. AKP Classics. A K
Peters, Ltd., Wellesley, MA, 1997.

\bibitem {14}F. W. J. Olver, D. W. Lozier, R. Boisvert, and C. W. Clark, eds.,
\textit{NIST Handbook of Mathematical Functions}, Cambridge Univ. Press,
Cambridge, 2010. Available at http://dlmf.nist.gov/

\bibitem {15}J. Sesma, Derivatives with respect to the order of the Bessel
function of the first kind. (2014), arXiv:1401.4850 [math.CA]

\bibitem {16}G. N. Watson, \textit{A Treatise on the Theory of Bessel
Functions}. 2nd edition, Cambridge University Press, Cambridge, England, 1944.
\end{thebibliography}
\end{document}